\documentclass[a4paper,reqno]{amsart}
\usepackage{amssymb,amsmath,amsthm,bbm,dsfont,mathrsfs,graphicx,subfigure}
\usepackage[all]{xy} 
\usepackage{pstricks,pstricks-add} 
\usepackage{epsfig}
\usepackage{hyperref}

\newtheorem{mthm}{Theorem}

\newtheorem{thm}{Theorem}[section]
 
 \newtheorem{lem}[thm]{Lemma}
 \newtheorem{prop}[thm]{Proposition}

 \theoremstyle{definition}
 \newtheorem{defn}[thm]{Definition}
 \theoremstyle{remark}

 \newtheorem{ex}[thm]{Example}
 
 \newcommand{\C}{\mathbb{C} }
 \newcommand{\CC}{\overline{\C} }

 \newcommand{\D}{\mathbb{D} }
 \newcommand{\Z}{\mathbb{Z} }

 \newcommand{\A}{\mathcal{A} }
 
 \newcommand{\Perp}{\perp \! \! \! \perp}

\begin{document}
 
\title{Thurston equivalence for rational maps with clusters}
\author{Thomas Sharland}
\address{University of Warwick, Coventry, CV4 7AL, UK}
\email{tomkhfc@hotmail.com}
\date{\today}
\subjclass[2010]{Primary 37F10}

\begin{abstract}
 We investigate rational maps with period one and two cluster cycles. Given the definition of a cluster, we show that, in the case where the degree is $d$ and the cluster is fixed, the Thurston class of a rational map is fixed by the combinatorial rotation number $\rho$ and the critical displacement $\delta$ of the cluster cycle. The same result will also be proved in the case that the rational map is quadratic and has a period two cluster cycle, but that the statement is no longer true in the higher degree case.
\end{abstract}
\maketitle

\tableofcontents

\section{Introduction}


Complex dynamics emergence as a popular subject for mathematical research came about as a result of the rediscovery of the early 20th Century works of Fatou \cite{Fatou:1919,Fatou:1920} and Julia \cite{Julia:1918}. After a comparably quiet period, the subject was given a new life in the 1980s. Perhaps the most notable contributions were supplied by A.~Douady and J.~Hubbard, whose ``Orsay lecture notes'' \cite{DouadyHubbard:Orsay1,DouadyHubbard:Orsay2} provide a number of enlightening and amazing results about the behaviour of such systems. Since then, the study of complex dynamical systems has grown enormously and is now a very fruitful area for research.

 In dynamical systems, one often wants to be able to understand the systems one works with up to some form of equivalence. The study of complex dynamics on the Riemann sphere is no different, and we are fortunate that there is a powerful criterion, due to Thurston, that tells us whether or not two rational maps on the sphere are equivalent. However, despite its simplicity and power, the criterion suffers from the fact it can, in practice, be very difficult to check. In this paper, we investigate a specific class of maps; those bicritical rational maps with periodic cluster cycles of period one or period two, and show that in this cases, the application of the Thurston criterion is simple and allows a classification of such maps. This paper was created from results from the author's PhD thesis at the University of Warwick \cite{Mythesis}. A second paper \cite{Matingspaper}, focusing on matings, is in preparation.

\subsection{Definitions}

Let $f \colon \CC \to \CC$ be a rational map on the Riemann sphere, of degree at least 2. The Julia set $J(f)$ will be the closure of the set of repelling periodic points of $f$, and the Fatou set is the set $F(f) = \CC \setminus J(f)$. The connected components of $F(f)$ are called Fatou components. In the case where the critical orbits are periodic, we call the immediate basins of the (super)attracting orbit critical orbit Fatou components. These components will play an important part in the definition of clustering, defined below.

We first discuss what it will mean for two rational maps to be equivalent. We use the standard definition of equivalence, which is Thurston equivalence. We will be dealing with postcritically finite rational maps. To define this, let $\Omega_F$ be the set of critical points of $F$. Then we define the postcritical set to be
\[
 P_{F} := \bigcup_{n>0} F^{\circ n} (\Omega_{F})
\]
and say the map $F$ is postcritically finite precisely when $|P_F|<\infty$.

\begin{defn}
Two postcritically finite rational maps $F$ and $G$ with labelled critical points will be called (Thurston) equivalent if there exists orientation preserving homeomorphisms $\phi_{1}, \phi_{2} \colon S^{2} \to S^{2}$ such that
   \begin{itemize}
     \item{$\phi_{1} |_{P_{F}} = \phi_{2} |_{P_{F}}$}
     \item{The following diagram commutes.
        \[
             \xymatrix{       (S^{2}, P_{F}) \ar[rr]^{\phi_{1}} \ar[dd]_{F}    && (S^{2}, P_{G}) \ar[dd]^{G}
            \\ \\
                (S^{2}, P_{F}) \ar[rr]_{\phi_{2}}                       && (S^{2},
                P_{G})}
        \]}
    \item{$\phi_{1}$ and $\phi_{2}$ are isotopic via homeomorphisms $\phi_{t}$, $t \in [0,1]$ satisfying $\phi_{0} |_{P_{F}} = \phi_{t} |_{P_{F}} = \phi_{1} |_{P_{F}}$ for each $t \in [0,1].$}
        \end{itemize}
\end{defn}

By a result of Thurston (for a proof, see \cite{DouadyHubbard:1993}), if two rational maps are equivalent then they are the same map up to conjugacy by a M\"{o}bius transformation. More precisely, the full theorem shows that each equivalence class of branched coverings of the sphere contains at most one rational map, up to M\"{o}bius conjugacy. Since we will be dealing with bicritical rational maps, the criterion for a branched covering to be equivalent to a rational map is a lot simpler (see \cite{TanLei:1990}). This is because, instead of needing to find Thurston obstructions, one can restrict the search to looking for Levy cycles. Let $\Gamma = \{ \gamma_{1}, \gamma_{2}, \ldots , \gamma_{n} \}$ be a collection of curves in $S^{2}$. If the $\gamma_{i} \in \Gamma$ are simple, closed, non-peripheral\footnote{Non-peripheral means $\gamma \cap P_{F} = \varnothing$ and each connected component of $S^{2} \setminus \gamma$ contains at least two points of $P_{F}$.}, disjoint and non-homotopic relative to $P_F$ then we say $\Gamma$ is a \emph{multicurve}.

\begin{defn}
  A multicurve $\Gamma = \{ \gamma_{1}, \gamma_{2}, \ldots , \gamma_{n} \}$ is a \emph{Levy cycle} if for each $i =1,\ldots,n$, the curve $\gamma_{i-1}$ (or $\gamma_{n}$ if $i = 1$) is homotopic to some component $\gamma_{i}'$ of $F^{-1}(\gamma_{i})$ (rel $P_{F}$) and the map $F \colon \gamma_{i}' \to \gamma_{i}$ is a homeomorphism.
\end{defn}

\begin{prop}[\cite{TanLei:1990}]
 Suppose $F$ is a bicritical branched covering. Then $F$ is (Thurston) equivalent to a rational map if and only if it does not have a Levy cycle.
\end{prop}

The concept of a cluster cycle for a bicritical rational map does not appear to be in the literature. Informally, it is the condition that the critical orbit Fatou components meet at a common boundary point, which is a repelling periodic point. This common boundary point will be called the cluster point. We give a more formal definition below. Recall that a bicritical rational map is said to be of type $D$ if the two critical points belong to the attracting basins of two disjoint periodic orbits.

\begin{defn}
	Let $F \colon \CC \to \CC$ be a bicritical rational map of type $D$ with the property that the two critical orbits belong to superattracting orbits with the same period. Then a cluster point for $F$ is a point in $J(F)$ which is the endpoint of the angle $0$ internal rays of at least one critical orbit Fatou component from each of the two critical cycles.

We will define a cluster to be the union of the cluster point and the Fatou components meeting at it. The period of the cluster will be the period of the cluster point.
	
The star of a cluster will be the union of the cluster point and the associated $0$ internal rays, including the points on the critical orbit.
\end{defn}

Since this definition is new, we give a simple example.

\begin{ex}
	Consider the mating of Douady's rabbit and the aeroplane polynomial (both polynomials have a superattracting orbit of period 3, and are post-critically finite). The result of this mating can be seen in Figure~\ref{f:2critclust}. This map has a fixed cluster point. In the figure, the basins of one of the critical points is white, the other basin is black. The cluster point is in the bottom right of the picture; its pre-image in the top left.
	
\begin{figure}[ht]
    \begin{center}
    \includegraphics[width=0.9\textwidth]{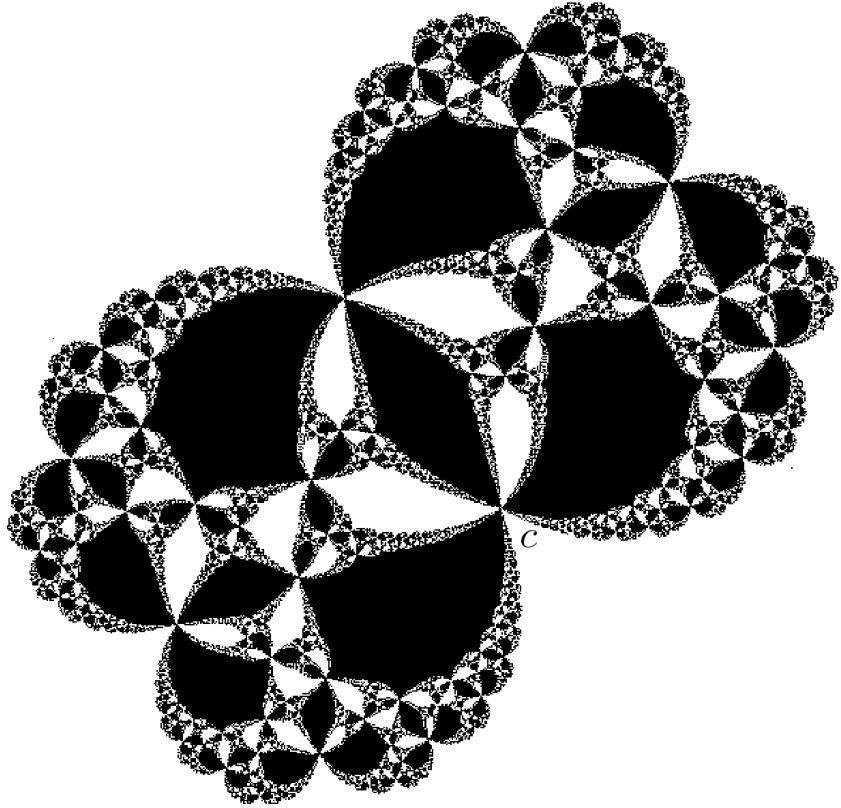}
    \end{center}
    \caption{An example of a map with a cluster point, labelled with $c$.}
    \label{f:2critclust}
\end{figure}	
\end{ex}

It is clear that a cluster will be invariant under the first return map, and hence it is possible to define the combinatorial rotation number of a cluster cycle in the usual way. More precisely, we note that the first return map to the cluster is actually a homeomorphism; this is because the $0$-internal rays map homeomorphically onto their images (which are also $0$-internal rays) under the rational map.

\begin{defn}\label{d:CRN}
    Let $F \colon \CC \to \CC$ be a rational map of type $D$. Let $c$ be a cluster point of period $n$ of $F$. Then the combinatorial rotation number is defined as follows. The first return map, $F^{\circ n}$, maps the star of the cluster, $X_{F}$ to itself. Label the arms of the star (the $0$-internal rays) which belong to one of the critical orbits (it does not matter which) cyclically in anticlockwise order by $\ell_1$, $\ell_2, \ldots \ell_n$ (the initial choice of $\ell_1$ is not important). Then for each $k$, there exists $p$ such that $F^{\circ n}$ maps $\ell_{k}$ to $\ell_{k+p}$, subscripts taken modulo $n$. We then say the combinatorial rotation number is $\rho = \rho(F) = p/n$. 
\end{defn}

There is another piece of data that we will require in this paper. This is the notion of the critical displacement. The definition we will use is different depending on whether we are discussing the period one or period two case. First of all, we need a lemma. The author is grateful to Mary Rees for suggesting the proof of the following result.

\begin{lem}\label{p:Marysresult}
	There does not exist a rational map $F$ with a period 2 cluster cycle such that the critical points are in the same cluster.
\end{lem}

\begin{proof}
We begin with some notation. We suppose that $F$ is a branched cover and the critical points are in the same cluster. Denote the critical points by $c_{0}$ and $\tilde{c}_{0}$, and denote $c_{i} = F^{\circ i}(c_{0})$ and $\tilde{c}_{j} = F^{\circ j}(\tilde{c}_{0})$. We will set both critical points to have period $2n$, and so $P_{F} = \{ c_{0},\ldots,c_{2n-1},\tilde{c}_{0},\ldots,\tilde{c}_{2n-1} \}$. Since the critical points are in the same cluster and the clusters are period 2, the set of post-critical points in the first cluster $\mathcal{C}_{0}$ are those of the form $c_{2i}$ and $\tilde{c}_{2j}$ (i.e those points with even index), whilst the remaining post-critical points lie in the second cluster $\mathcal{C}_{1}$. Denote the star of $\mathcal{C}_0$ by $X_0$ and the star of $\mathcal{C}_1$ by $X_1$.
	
Now consider the curve $\gamma$, the boundary of a tubular neighbourhood $U$ of $X_0$. Then $\gamma$ separates the two clusters and in particular is non-peripheral. By Lemma 3.4 in \cite{Milnor:Twocrit}, $F^{-1}(\gamma)$ is made up of $d$ disjoint curves, each of which is the boundary of a tubular neighbourhood of a pre-image star of $X_0$. If the pre-image star is not $X_1$, then the curve in $F^{-1}(\gamma)$ bounding its tubular neighbourhood must be peripheral. The boundary of the tubular neighbourhood of $X_1$, $\gamma' \in F^{-1}(\gamma)$ separates the two clusters, and so is isotopic to $\gamma$. $F : \gamma' \to \gamma$ is a homeomorphism, hence $\Gamma = \{ \gamma \}$ is a Levy cycle and hence such a branched cover cannot be equivalent to a rational map. 
\end{proof}

The reader will have noted that the definition of the critical displacement is dependent on the choice of which critical point is chosen to be the first one. Hence our results us to require us to study rational maps with labelled critical points, so that the choice of the first critical point is known.

Informally, we want the critical diplacement to tell us how far apart the critical points are in the clusters. In light of Lemma~\ref{p:Marysresult}, we will want to define the critical displacement depending on whether we are in the fixed case or the period two case. In the fixed case, this will be easy: we can just calculate the (combinatorial) distance between the two critical points around the cluster. Clearly this is not possible in the period two case (or indeed, in any case where the period of the cluster is greater than 1). So we modify the definition so that we now measure the (combinatorial) distance around the cluster between the first critical point and the image of the second critical point. The formal definitions are below.

\begin{defn}
 Let $F$ be a rational map with a fixed cluster point. Label the endpoints of the star as follows. Let $e_0$ be the first critical point, and label the remaining arms in anticlockwise order by $e_1,e_2,\ldots,e_{2n-1}$. Then the second critical point is one of the $e_j$, and we call $j$ the critical displacement of the cluster of $F$. We denote the critical displacement by $\delta$.
\end{defn}

\begin{defn}
  Let $F$ be a rational map with a period two cluster cycle. Choose one of the critical points to be $c_1$, and label the cluster containing it to be $\mathcal{C}_{1}$. Then (by Lemma~\ref{p:Marysresult}) the other critical point $c_2$ is in the second cluster $\mathcal{C}_2$. We define the critical displacement $\delta$ as follows. Label the arms in the star of $\mathcal{C}_1$, starting with the arm with endpoint $c_1$, in anticlockwise order $\ell_0, \ell_1, \ldots, \ell_{2n-1}$. Then $F(c_2)$ is the endpoint of one of the $\ell_k$. This integer $k$ is the critical displacement.
\end{defn}

The critical displacement will always be an odd integer, since the critical orbit Fatou components alternate around a cluster. That is, a Fatou component provided from the first critical point $c_1$ must be between two Fatou components provided by the orbit of the second critical point $c_2$. We can now define the combinatorial data of a map with a cluster cycle to be the pair $(\rho,\delta)$. Note that this data is intrinsic to the cluster - it does not \emph{a priori} give any information about the rational map away from the cluster. 

We are now ready to state the two main theorems. Essentially, they both state that, in all degrees for the fixed case and in the quadratic case for period two cluster cycles, the combinatorial data is enough to define the rational map in question up to conjugacy by a M\"{o}bius transformation.

\begin{mthm}\label{Fixedcase}
  Suppose $F$ and $G$ are bicritical rational maps (with labelled critical points) with fixed cluster cycles with the same combinatorial data. Then $F$ and $G$ are equivalent in the sense of Thurston.
\end{mthm}

\begin{mthm}\label{Per2case}
  Suppose that two quadratic rational maps $F$ and $G$ have a period two cluster cycle with rotation number $p/n$ and critical displacement $\delta$. Then $F$ and $G$ are equivalent in the sense of Thurston.
\end{mthm}

The proofs of these theorems are in Sections~\ref{s:fixed} and~\ref{s:per2} respectively. In Section~\ref{s:higher}, we will also show that (conjecturally) there exists degree 3 rational maps with the same combinatorial data which are not equivalent in the sense of Thurston. This means that, to carry out a classification in higher degrees, we need to find some combinatorial data that is extrinsic to the cluster - data which somehow is independent of the behaviour of the cluster cycle.

\section{Thurston equivalence for the fixed case}\label{s:fixed}

We now prove Theorem~\ref{Fixedcase}. The proof is relatively simple but the techniques used in proving it will make a useful comparison with the difficulties encountered when trying to prove the period two case in the next section. The proof will proceed as follows. To prove Thurston equivalence, we need to find homeomorphisms $\Phi,\widehat{\Phi} \colon \CC \to \CC$ which satisfy
	\begin{enumerate}
		\item{$\Phi \circ F = G \circ \widehat{\Phi}$.}
		\item{$\Phi|_{P_F} = \widehat{\Phi}|_{P_F}$.}
		\item{$\Phi$ and $\widehat{\Phi}$ are isotopic rel $P_F$.}
	\end{enumerate}
We will first construct the homeomorphism $\Phi$. We then try to construct the homeomorphism $\widehat{\Phi}$ so that it satisfies the conditions 1 to 3 above. The first two conditions will be satisfied by the construction given, whilst the third will follow from an application of Alexander's Trick. The whole proof will be broken down into a sequence of lemmas.

In what follows, we will denote the stars of $F$ and $G$ by $X_F$ and $X_G$ respectively. Recall that the star $X_F$ of a rational map $F$ is made up of the union of the internal rays inside the critical orbit Fatou components and the cluster point. By B\"{o}ttcher's theorem, the dynamics of the first return map on each critical orbit Fatou component is then conjugate to the map $z \mapsto z^{d}$ on the disk, $\D$. We also can label the critical orbit points cyclically as follows. Let $c_{0}$ be the first critical point, in terms of the ordering induced by the critical displacement, so that the critical displacement is defined to be the combinatorial distance (anticlockwise) around the star from $c_0$ to the other critical point. We label, counting anticlockwise, the other critical orbit points by $c_1, c_2, \ldots,c_{2n-1}$, and denote the Fatou component containing $c_i$ by $U_i$. Note that at this point we are not worried about which critical orbit the $c_i$ are in, since we are only concerned with the dynamics of the first return map on each component. Finally let the $0$-internal ray in $U_i$ be labelled $J_i$.

In the following lemma, the objects associated with the map $G$ will be given a $'$ to differentiate them from the objects associated with $F$. For example, the first critical point of $X_G$ will be labelled $c_{0}'$, and it will be in the Fatou component $U_0'$.

\begin{lem}\label{lemma1}
	Suppose $F$ and $G$ have the same combinatorial data. Then there exists a conjugacy $\phi \colon X_F \to X_G$. That is, $\phi \circ F = G \circ \phi$ on $X_F$. Furthermore, the conjugacy can be constructed so as to preserve the cyclic ordering of the internal rays in the star.
\end{lem}

\begin{proof}
We will show that there is a conjugacy between the dynamics on $J_i$ and $J_i'$ for each choice of $i$. There is a map $h_{F,i}$ conjugating the dynamics on $U_i$ with that of $z \mapsto z^d$ on $\D$, so that $h_{F,i}(J_i) = [0,1)$. Similarly, there exists a conjugacy $h_{G,i}$ from $U_i'$ to $z \mapsto z^d$ on $\D$, with $h_{G,i}(J_i) = [0,1)$ . So the map $\phi_i = h_{G,i}^{-1} \circ h_{F,i}$ conjugates the dynamics on $U_i$ with that on $U_i'$, and in particular takes $J_i$ to $J_i'$. So the restriction of $\phi_i$ to $J_i$ is the required conjugacy on $J_i$.

The required conjugacy $\phi$ is then defined by mapping the cluster point $c \in X_F$ to the cluster point $c' \in X_G$, (i.e, $\phi(c) = c'$) and then picking $\phi|_{J_i} = \phi_i$.
\end{proof}

\begin{lem}\label{lemma3}
	Let $\phi$ be the homeomorphism from Lemma~\ref{lemma1}. Then there exists continuous maps $\tilde{\eta}_F$ and $\tilde{\eta}_G$ and a homeomorphism $\psi$ such that the following diagram commutes.
	\[
             \xymatrix{       \partial \D \ar[rr]^{\psi} \ar[dd]_{\tilde{\eta}_F}    && \partial \D \ar[dd]^{\tilde{\eta}_G}
            \\ \\
                X_F \ar[rr]_{\phi}                       && X_G }
        \]
\end{lem}

\begin{proof}
We remark that $\CC \setminus X_F$ is simply connected, since $X_F$ is a connected set. Hence, by the Riemann Mapping Theorem, there exists a Riemann map $\eta_F \colon \CC \setminus \overline{\D} \to \CC \setminus X_F$. Similarly, there exists a Riemann map $\eta_G \colon \CC \setminus \overline{\D} \to \CC \setminus X_G$. Since the star $X_F$ is locally connected, by Carath\'{e}odory's Theorem, we can extend the maps $\eta_F$ and $\eta_G$ to $\CC \setminus \D$ in a continuous way. We label these extensions $\tilde{\eta}_F$ and $\tilde{\eta}_G$.

As $\phi$ is a homeomorphism, it maps arms of the star $X_F$ to arms of $X_G$. We now define the map $\psi$. Clearly, we would like to define $\psi = \tilde{\eta}_G^{-1} \circ \phi \circ \tilde{\eta}_F$. However, since most points in $X_G$ (indeed, all points not in $P_G$) have more than one pre-image under the mapping $\tilde{\eta}_G$, this function is not well-defined. However, it is possible to use this motivating idea to construct $\psi$, by choosing the ``correct" pre-image when necessary. Note that each point in the postcritical set of $F$ has precisely one pre-image under $\tilde{\eta}_F$, and the cyclic ordering of the $\tilde{\eta}_F^{-1}(c_{i})$ is the same as the cyclic ordering of the $c_{i}$. Since the same is true for $\tilde{\eta}_G$ (that is, points on the postcritical set have only one pre-image), we can define $\psi = \tilde{\eta}_G^{-1} \circ \phi \circ \tilde{\eta}_F$ on the postcritical set. Since $\phi$ will rotate the arms of the star by $k$ places anticlockwise, $\psi$ will do the same. For ease of notation we will write $p_{i} = \tilde{\eta}_{F}^{-1}(c_{i})$ and $p_{i}' = \tilde{\eta}_{F}^{-1}(c_{i}')$. With this new notation, therefore, we have $\psi(p_{i}) = p_{i}'$.

The cluster point $c$ in $X_F$ has $2n$ preimages under $\tilde{\eta}_F$, and each pre-image lies in one of the arcs $(p_{i},p_{i+1})$ for $i = 0,\ldots,2n-1$ (otherwise cyclic ordering would not be maintained). Label the pre-image in $(p_{i},p_{i+1})$ by $\xi_{i}$. Similarly, the cluster point $c'$ in $X_G$ has $2n$ preimages under $\tilde{\eta}_G$, and each pre-image lies in one of the arcs $(c_{i}',c_{i+1}')$ for $i = 0,\ldots,2n-1$. So we denote the pre-image in $(c_{i}',c_{i+1}')$ by $\xi_{i}'$. We then define $\psi(\xi_{i}) = \xi_{i}'$, and note that this satisfies $\tilde{\eta}_G \circ \psi =  \phi \circ \tilde{\eta}_F$, since $\phi(c) = c'$. Furthermore, this agrees with the cyclic ordering induced on the circle by the rotation of arms in the map $\phi$. See Figure~\ref{f:stardisk} for the construction so far in the case where the critical orbits have period $4$.

\begin{figure}[ht]
\begin{center}

\input{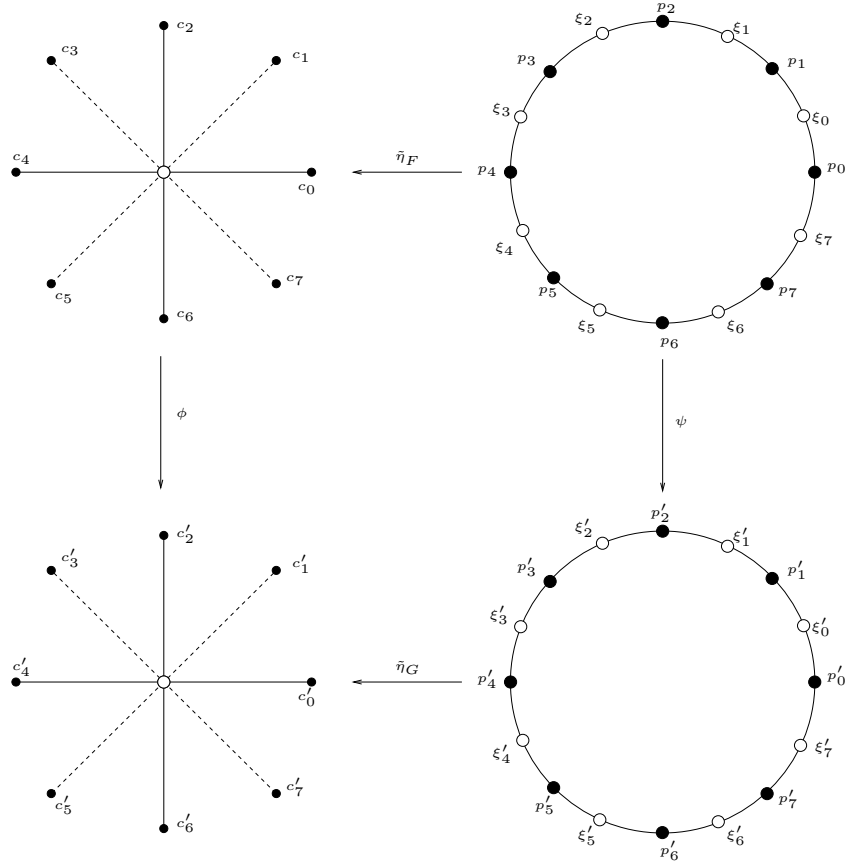}

\caption{Construction of the map $\psi$ in Lemma~\ref{lemma3}.}
\label{f:stardisk}
\end{center}
\end{figure}

Now let $z \in X_F$ where $z$ is not in the postcritical set or equal to the cluster point. Then $z$ has precisely two pre-images under $\tilde{\eta}_F^{-1}$ and the point $z' = \phi(z) \in X_G$ has two pre-images under $\tilde{\eta}_G^{-1}$. For the diagram in the statement of the lemma to commute, we need to have $\psi(\tilde{\eta}_F^{-1}(z)) \in \tilde{\eta}_G^{-1}(z')$. It is clear that $z$ must belong to the interior of some internal ray of the form $[0,c_{i}] \subset X_F$, hence $z'$ belongs to the internal ray $[0,c_{i}'] \subset X_G$. This means that there is a point $w_{1}$ of $\tilde{\eta}_F^{-1}(z)$ in the arc $(\xi_{i-1},p_{i})$ and the other point $w_{2}$ must be in the arc $(p_{i},\xi_{i})$. Furthermore, the two pre-images of $z'$ (under the map $\tilde{\eta}_G$) are $w_{1}' \in (\xi_{i-1}',p_{i}')$ and $w_{2}' \in (p_{i}',\xi_{i})$. We now define $\psi(w_{1}) = w_{1}'$ and $\psi(w_{2}) = w_{2}'$. This definition satisfies $\phi \circ \tilde{\eta}_F = \tilde{\eta}_G \circ \psi$. Notice further that this construction will preserve the cyclic ordering of the points on the circle.

We now show that $\psi$ is a homeomorphism. The construction of $\psi$ shows it is clearly bijective, so we only need to show $\psi$ is continuous. Then, since $\psi$ will be a continuous bijection from a compact space to a Hausdorff space, it will be a homeomorphism  Indeed, since $\tilde{\eta}_F$ and $\phi$ are continuous maps, it is sufficient to check the choices we made for $\tilde{\eta}_G^{-1}$ are done in a continuous way. We have three cases.

\noindent\emph{Case 1:} $z \notin \{ p_{0},\ldots,p_{2n-1},\xi_{0},\ldots,\xi_{2n-1} \}$. Suppose $z$ lies in some arc $(p_{i}, \xi_{i})$. Then any sequence $v_{n} \to z$ will belong to $(p_{i}, \xi_{i})$ if $n$ is large enough. Then by construction we will have $\psi(v_{n}) \in (p_{i}',\xi_{i}')$. Now we notice that $\tilde{\eta}_G$ is a homeomorphism on $(p_{i}', \xi_{i}')$ and so we must have $\psi$ is continuous at $z$ as it is (locally) the composition of continuous maps. A similar argument holds when $z$ belongs to some arc $(\xi_i,p_{i+1})$.

\noindent\emph{Case 2:} $z = p_{i}$. Let $x_{n} \to z$. Then $\phi(\tilde{\eta}_F(x_n)) \in (0,c_{i}']$ for $n$ sufficiently large and by continuity of $\tilde{\eta}_F$ and $\phi$, $\phi(\tilde{\eta}_F(x_n)) \to c_{i}'$. Now the map $\tilde{\eta}_G$ is a branched covering when restricted to $(\xi_{i-1}',\xi_{i}')$, with image $(0,c_{i}']$. The unique branch point is $p_{i}'$ which has its image at $c_{i}'$. So as we have $\phi(\tilde{\eta}_F(x_n)) \to c_{i}'$, we must have $\tilde{\eta}_G^{-1}(\phi(\tilde{\eta}_F(x_n))) \to \tilde{\eta}_G^{-1}(c_{i}') = p_{i}'$. 

\noindent\emph{Case 3:} $z = \xi_{i}$. Let $y_{n} \to z$. Then for $n$ sufficiently large, $y_{n} \in (p_{i},p_{i+1})$, $\phi(\tilde{\eta}_F(y_n)) \in [c',c_{i}')\cup[c',c_{i+1}))$ and $\phi(\tilde{\eta}_F(x_n)) \to c'$. So $\tilde{\eta}_G^{-1}(\phi(\tilde{\eta}_F(y_{n}))) \in (p_{i}',p_{i+1}')$, by the construction of $\psi$ (the other pre-images of $\phi(\tilde{\eta}_F(y_{n}))$ cannot be the images of $\psi$, since as $y_{n} \in (p_{i},p_{i+1})$, we must have $\psi(y_{n}) \in (p_{i}',p_{i+1}')$). We then use the fact that $\tilde{\eta}_G$ restricted to $(p_{i}',p_{i+1}')$ is a homeomorphism, which once again means $\psi$ is continuous at $z$. Hence $\phi$ is a homeomorphism.
\end{proof}

We now extend the map $\psi$.

\begin{prop}\label{p:extension}
	The map $\psi$ of Lemma~\ref{lemma3} can be extended to a homeomorphism $\Psi \colon \CC \setminus \D \to \CC \setminus \D$. This map $\Psi$ induces a homeomorphism $\Phi \colon \CC \to \CC$, such that $\Phi|_{X_F} = \phi$ (i.e, $\Phi$ is an extension of $\phi$ to the sphere) and
	\[
             \xymatrix{       \CC \setminus \D \ar[rr]^{\Psi} \ar[dd]_{\tilde{\eta}_F}    && \CC \setminus \D \ar[dd]^{\tilde{\eta}_G}
            \\ \\
                \CC \ar[rr]_{\Phi}                       && \CC }
        \]
commutes.
\end{prop}

\begin{proof}
The extension of $\psi$ to $\Psi$ is an application of the Alexander Trick. 

We want $\Phi$ to be an extension of $\phi$, and hence it is necessary that we have $\Phi(z) = \phi(z)$ on $X_F$. Note that, considering $\tilde{\eta}_F^{-1}(z)$ as a set, the commutative diagram for $\psi$ in Lemma~\ref{lemma3} suggests we can write $\phi(z) = (\tilde{\eta}_G \circ \psi) ( \tilde{\eta}_F^{-1}(z) )$ for $z \in X_F$.  Bearing this in mind, define
\[
	\Phi(z)  = \left\{
                        \begin{array}{ll}
                            \eta_G \circ \Psi \circ \eta_F^{-1}(z) ,  & \hbox{$z \in \CC \setminus X_F$;} \\
                            \phi (z),                  & \hbox{$z \in X_F$.}
                        \end{array}
             \right.
\]
The proof that $\Phi$ is a homeomorphism is essentially just using the fact that $\tilde{\eta}_F$ and $\tilde{\eta}_G$ are quotient maps. This means that $\tilde{\eta_G} \circ \Psi$ is a quotient map and so $\Phi$ is a homeomorphism.
\end{proof}

We note that, so far in this section, we have not needed any requirement about the combinatorial data being equal, so these results hold in full generality. However, the next result is the point where the equality of combinatorial data is needed. 

Before we prove the next proposition, we briefly discuss the space $\CC \setminus X_F$ and its pre-image $\CC \setminus F^{-1}(X_F)$. Informally, first note that the star $X_F$  has $d$ pre-images in $\CC$, and each pre-image is disjoint, save for the fact that they all contain the critical points of $F$. More exactly, we notice that the set $X_F \setminus \{\text{critical values}\}$ has $d$ disjoint pre-images under $F$, and the union of one of these pre-images with the two critical points will map homeomorphically onto $X_F$. We call each of these pre-images, with its union with the critical points, a pre-image star of $X_F$. Note that there is a cyclic order of these pre-image stars at $c_0$ (the first marked critical point), so that we can label them as follows. Label $X_F$ by $X_1$. Counting anticlockwise from this pre-image, label the remaining stars $X_2, X_3, \ldots X_d$.

Furthermore, we notice that $\CC \setminus F^{-1}(X_F) = \CC \setminus \bigcup_{j=1}^{d} X_j$ contains $d$ connected components. Label these components in anticlockwise order, starting with $\A_1$ as the component anticlockwise from $X_F = X_1$ with respect to the cyclic ordering at $c_{0}$, and the others in order as $\A_2,\ldots,\A_d$. We remark that in this notation, the boundary of $\A_j$ is contained in $X_j \cup X_{j+1}$. With this notation,  the map $F|_{\A_j} \colon \A_j \to \CC \setminus X_F$ is a homeomorphism, and the map $F|_{\CC \setminus F^{-1}(X_F)} \colon \CC \setminus F^{-1}(X_F) \to \CC \setminus X_F$ is a degree $d$ covering map. For ease of notation we write $F_j = F|_{\A_j}$. See Figure~\ref{f:mainthmpic}.
\begin{figure}[ht]
\begin{center}

\input{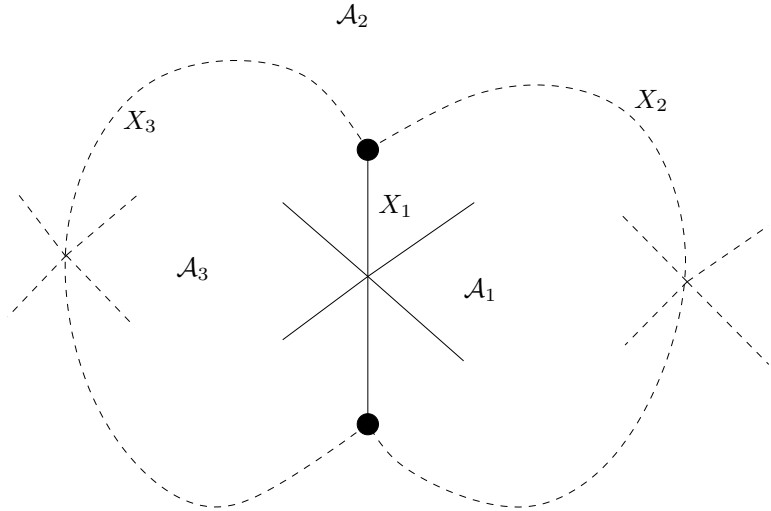}

\caption{The star (bold line) and pre-image stars (dashed line) and how they separate the sphere. The black dots represent the critical points, where the star and pre-stars meet.}
\label{f:mainthmpic}
\end{center}
\end{figure}
We can carry out a similar construction with $G$. Using the same construction as above, the pre-image stars of $X_G$ are $X_1',X_2',\ldots,X_d'$ and the connected components of $\CC \setminus G^{-1}(X_G)$ are $\A_1',\A_2',\ldots,\A_d'$. $G_j$ will denote the map $G|_{\A_j'}$.

\begin{prop}\label{newPhi}
	There exists a homeomorphism $\widehat{\Phi} \colon \CC \to \CC$ so that
	\[
             \xymatrix{       \CC \ar[rr]^{\widehat{\Phi}} \ar[dd]_{F}    && \CC \ar[dd]^{G}
            \\ \\
                \CC \ar[rr]_{\Phi}                       && \CC }
        \]
commutes.
\end{prop}

\begin{proof}
As with Lemma~\ref{lemma3}, this proof is constructive. First note that if $z \in X_F$, then we define $\widehat{\Phi}(z) = \phi(z) \in X_G$, taking advantage of the fact that $\phi$ is a conjugacy between the dynamics on $X_F$ and $X_G$ (Lemma~\ref{lemma1}).
	
Now suppose $z$ is in $F^{-1}(X_F)$. The case where $z \in X_F$ is dealt with above, so we may assume $z \in X_j$ for some $j \in \{2,\ldots,d \}$. Then for the diagram to commute we require $\widehat{\Phi}(z) \in G^{-1} \circ \Phi \circ F(z) = G^{-1} \circ \phi \circ F(z)$. The set $G^{-1} \circ \Phi \circ F(z)$ contains $d$ elements, one each in $X_1',\ldots,X_d'$. Since $z \in X_j$, we choose $\widehat{\Phi}(z) \in X_{j}'$.

Finally, suppose $z \in \CC \setminus F^{-1}(X_F)$. Then $z \in \A_j$ for some $j \in \{1,\ldots,d \}$. With a similar argument to that in the previous paragraph, the set $G^{-1} \circ \Phi \circ F(z)$ contains $d$ elements, one in each of the $\A_j$. So as we have, $z \in \A_j$, we define $\widehat{\Phi}(z)$ to be the element of $G^{-1} \circ \Phi \circ F(z)$ in $\A_j'$.

We now show $\widehat{\Phi}$ is a homeomorphism. Let $U$ be an open disc in $\CC$, which is disjoint from the critical points. We will show $\widehat{\Phi}^{-1}(U)$ is open. By commutativity, $\widehat{\Phi}^{-1}(U)$ is contained in $F^{-1}(\Phi^{-1}(G(U)))$. Since $G$ is a rational map, $G(U)$ is an open set, and by continuity of $\Phi$, $\Phi^{-1}(G(U))$ is also open. Furthermore, $\Phi^{-1}(G(U))$ is disjoint from the two critical values, and so $F^{-1}(\Phi^{-1}(G(U)))$ is made up of $d$ disjoint open sets. By the construction of $\widehat{\Phi}$ given above, we see that precisely one of these is the set $\widehat{\Phi}^{-1}(U)$, which is therefore open. This proves continuity for the non-critical points.

If $U$ is a disc which contains one critical point, then a similar argument shows that $F^{-1}(\Phi^{-1}(G(U)))$ is a single simply connected open set, and so $\widehat{\Phi}$ is continuous at the critical points. Hence $\widehat{\Phi}$ is a continuous bijection from a compact space to a Hausdorff space, and so is a homeomorphism. \end{proof}

\begin{lem}\label{newPsi}
There exists a homeomorphism $\widehat{\Psi} \colon \CC \setminus \D \to \CC \setminus \D$ so that
	\[
             \xymatrix{       \CC \setminus \D \ar[rr]^{\widehat{\Psi}} \ar[dd]_{\tilde{\eta}_F}    && \CC \setminus \D \ar[dd]^{\tilde{\eta}_G}
            \\ \\
                \CC \ar[rr]_{\widehat{\Phi}}                       && \CC }
        \]
commutes.
\end{lem}

\begin{proof}
This is analogous to the proof of Proposition~\ref{p:extension}. We define
	\[
	\widehat{\Psi}(z)  = \left\{
                        \begin{array}{ll}
                            \eta_G^{-1} \circ \widehat{\Phi} \circ \eta_F(z) ,  & \hbox{$z \in \CC \setminus \overline{\D}$;} \\
                            \psi (z),                  & \hbox{$z \in \partial \D$.}
                        \end{array}
             \right.
\]
Again, continuity on $\CC \setminus \overline{\D}$ is assured by the fact that $\Psi$ is defined as a composition of homeomorphisms there. So we only need to check continuity on the boundary, $\partial \D$.

The cyclic ordering induced by the Riemann maps and $\widehat{\Phi}$ is induced by the cyclic ordering from $\phi$ and $\psi$. Notice, if $x_n \to x \in X_F$, $\lim_{n \to \infty}(\widehat{\Phi}(x_{n})) = \phi(x)$. Also $\psi = \widehat{\Psi}|_{\partial \D}$ is chosen as the homeomorphism of the circle which satisfies $\phi \circ \tilde{\eta}_F(z) = \tilde{\eta}_G \circ \psi(z)$ for all $z$. This means that $\psi(z)$ is the element of $\tilde{\eta}_G(\tilde{\eta}_G^{-1} \circ \widehat{\Phi} \circ \tilde{\eta}_F(z))$ which maintains the cyclic ordering of the points. Hence any sequence converging to $z$ must converge to $\psi(z)$ under $\widehat{\Psi}$ (else we would lose the ordering), and so the given boundary values for $\widehat{\Psi}$ give continuity. Once again, since $\widehat{\Psi}$ is a continuous bijection from a compact space to a Hausdorff space, it is a homeomorphism. \end{proof}

We are now ready to prove Theorem~\ref{Fixedcase}.

\begin{proof}[Proof of Theorem~\ref{Fixedcase}]
 We claim that we have the following commutative diagram.
\[
		\xymatrix{
				X_{F} \ar[ddrr]^{F} \ar[rrrrrr]^{\phi} &&&&&& X_G  \ar[ddll]_G \\
				&&& *+[o][F-]{1}  &&&\\
				&&	X_F \ar[rr]^{\phi} && X_G &&\\ 
				&  & & *+[o][F-]{2} & &&   \\
				&&	\partial \D \ar[uu]^{\tilde{\eta}_{F}} \ar[rr]_{\psi} && \partial \D \ar[uu]_{\tilde{\eta}_G} && \\
				&&& *+[o][F-]{3} &&& \\
				\partial \D \ar[rrrrrr]_{\psi} \ar[uuuuuu]^{\tilde{\eta}_F} &&&&&& \partial \D \ar[uuuuuu]_{\tilde{\eta}_G}  } 
\]
Each part of this diagram is justified as follows.	
	\begin{enumerate}
		\item{$\phi$ is a conjugacy so $\phi \circ F = G \circ \phi$. The existence of $\phi$ is given by Lemma~\ref{lemma1}.}
		\item{Lemma~\ref{lemma3}}
		\item{Lemma~\ref{lemma3}.}
	\end{enumerate}

This diagram extends to give a commutative diagram:

\[
\xymatrix{
	(\CC,X_{F}) \ar[ddr]^{F} \ar[rrrr]^{(\widehat{\Phi},\phi)} &&&& (\CC,X_G) \ar[ddl]_G \\
	&& *+[o][F-]{1}  &&\\
	&	(\CC,X_F) \ar[rr]^{(\Phi,\phi)} && (\CC,X_G) &\\ 
	&   & *+[o][F-]{2} & &   \\
	&	(\CC \setminus \D,\partial \D) \ar[uu]^{\tilde{\eta}_{F}} \ar[rr]_{(\Psi,\psi)} && (\CC \setminus \D,\partial \D) \ar[uu]_{\tilde{\eta}_G} & \\
	&& *+[o][F-]{3} && \\
	(\CC \setminus \D,\partial \D) \ar[uuuuuu]^{\tilde{\eta}_F} \ar[rrrr]_{(\widehat{\Psi},\psi)} &&&& (\CC \setminus \D,\partial \D) \ar[uuuuuu]_{\tilde{\eta}_G}  }				
	\]
where the notation $(\Phi, \phi) \colon (\CC,X_F) \to (\CC,X_G)$ means the map  is defined as $\Phi$ on $\CC$, and its restriction to $X_F$ is $\phi$. The other maps are defined analogously.
This diagram is justified by
	\begin{enumerate}
		\item{Proposition~\ref{newPhi}.}
		\item{Proposition~\ref{p:extension}}
		\item{Lemma~\ref{newPsi}}
	\end{enumerate}
	
We now remark that the maps $\Phi$ and $\widehat{\Phi}$ agree on $X_F$, and so agree on the set $P_F \subset X_F$. Furthermore, $\Psi$ is isotopic to $\widehat{\Psi}$, by the Alexander trick, and so we see that the commutative diagram above (and the fact that $\tilde{\eta}_F$ and $\tilde{\eta}_G$ are homeomorphisms on $\CC \setminus \overline{\D}$) gives us that $\Phi$ and $\widehat{\Phi}$ are isotopic rel $X_F$ and so isotopic rel $P_F$. Hence $F$ and $G$ are Thurston equivalent.
\end{proof}

\section{Thurston equivalence for the period two case}\label{s:per2}

The case where the clusters are of period two is more difficult than the fixed case. In the previous section, we were able to use the Alexander Trick, which allowed us to get the isotopy between the two homeomorphisms $\Phi$ and $\widehat{\Phi}$. There are now two stars. Denote the star containing the first critical point of $F$ by $X_F^1$ and the one containing the second critical point $X_F^{2}$. We similarly define the stars $X_G^1$ and $X_G^2$. The following result is entirely analogous to Theorem~\ref{Per2case}. The added complexity comes about because we no longer have an analogue to the isotopy version of Alexander's Trick. Informally, this is because the space $\CC \setminus (X_1 \cup X_2)$ is no longer simply connected, and so is not conformally equivalent to a disk. Indeed, it is conformally equivalent to an annulus. The simplicity of Theorem~\ref{Per2case} resulted from the mapping class group of the disk being trivial. For the annulus, the mapping class group is now $\Z$, and so we see that there is an extra difficulty when we consider the period two cluster case.

We also emphasise that, unlike in the previous section, we only prove the result in the case that the rational maps $F$ and $G$ have degree $2$. There is strong evidence to suggest that, in fact, the theorem is not true in the higher degree case.. A number of the proofs in this section have proofs completely analogous to those in the previous section. We will proceed as follows. 

\begin{itemize}
  \item{First we will construct a pair of homeomorphisms $\Phi$ and $\widehat{\Phi}$ which satisfy $\Phi \circ F = G \circ \widehat{\Phi}$}
  \item{We then modify $\Phi$ and $\widehat{\Phi}$ to homeomorphisms $\Phi_1$ and $\widehat{\Phi}_1$ which still satisfy $\Phi \circ F = G \circ \widehat{\Phi}$ and also agree on the stars $X_F^1 \cup X_F^2$.}
  \item{We finally modify $\Phi_1$ and $\widehat{\Phi}_1$ to homeomorphisms $\Phi_2$ and $\widehat{\Phi}_2$ which still satisfy $\Phi \circ F = G \circ \widehat{\Phi}$, agree on the stars $X_F^1 \cup X_F^2$ and furthermore are isotopic rel $P_F$.}
\end{itemize}
 
This will mean the homeomorphisms $\Phi_2$ and $\widehat{\Phi}_2$ satisfy the conditions of the homeomorphisms in the definition of Thurston equivalence of $F$ and $G$, and so $F$ and $G$ will be equivalent. We begin with an analogue to Lemma~\ref{lemma1}.

\begin{lem}\label{lemma11}
	There exists a conjugacy $\phi \colon (X^1_F \cup X^2_F) \to (X^1_G \cup X^2_G)$ so that
\begin{equation}\label{e:phi}
 \phi \circ F = G \circ \phi.
\end{equation}
\end{lem}

\begin{proof}
	The proof of this lemma is essentially the same as that for Lemma~\ref{lemma1}. Notice in particular that we require the stars to be marked, so that we know which one contains the first critical point in the definition of combinatorial displacement.
\end{proof}

The next result is the first point where we notice a difference with the previous section. By the Riemann mapping theorem for simply connected regions (not equal to the whole of $\C$), we saw that the complement to the star in the sphere will be conformally isomorphic to the unit disk. However, in this case, we use the fact that the complement to the stars will be conformally equivalent to some annulus $A$. However, two annuli $A_1$ and $A_2$ are conformally equivalent to each other if and only if the ratio of the radii of their boundary circles are the same. In other words, if we normalise so that the radius of the inner boundary circle is 1, we see that two annuli are conformally equivalent if and only if their outer boundary circles have the same radii.

\begin{prop}
	Let $F$ be a rational map with a period two cluster cycle. Then there exists a conformal map $\eta_F \colon A_F \to \CC \setminus (X^1_F \cup X^2_F)$, where $A_F$ is an annulus. Furthermore, this conformal equivalence can be extended to a continuous function $\tilde{\eta}_F \colon \overline{A}_F \to \CC$.
\end{prop}

\begin{proof}
	The conformal equivalence with the annulus is a standard result and the continuous extension follows from the fact that the stars are locally connected.
\end{proof}

We now set our notation so that $A_F = \{ z \, \colon \, 1 < |z| < e^{R_F} \}$, $A_G = \{ z \, \colon \, 1  <   |z|  <  e^{R_G} \}$ and $\eta_F \colon A_F \to \CC \setminus ( X^1_F \cup X^2_F )$ and $\eta_G \colon A_G \to \CC \setminus (X^1_G \cup X^2_G)$ are conformal equivalences. 

\begin{prop}\label{p:Psi2}
	Let $\psi$ be a homeomorphism defined on the boundary of the annulus $A_F$ mapping to the boundary of the other annulus $A_G$, which preserves the orientation on each boundary circle. Then $\psi$ can be extended to a homeomorphism $\Psi \colon A_F \to A_G$.
\end{prop}

\begin{proof}
By the results of \cite{Youngs}, it is always possible to extend a homoeomorphism on the boundary of a 2-manifold to the whole manifold.
\end{proof}

\begin{lem}\label{l:Phi2}
	The homeomorphism $\Psi$ induces a homeomorphism $\Phi \colon \CC \to \CC$. Moreover, $\Phi |_{X^1_F \cup X^2_F} = \phi$.
\end{lem}

\begin{proof}
	We define
\[
	\Phi(z)  = \left\{
                        \begin{array}{ll}
                            \eta_G \circ \Psi \circ \eta_F^{-1}(z) ,  & \hbox{$z \in \CC \setminus (X^{1}_F \cup X^{2}_F)$;} \\
                            \phi (z),                  & \hbox{$z \in X^{1}_F \cup X^{2}_F$.}
                        \end{array}
             \right.
\]
Clearly $\Phi$ is a bijection, and it is a homeomorphism with a similar argument as in Proposition~\ref{p:extension}.
\end{proof}

We now define the homeomorphism $\widehat{\Phi}$ to be the lifting of $\Phi$ under $F$ and $G$; that is $\widehat{\Phi} = G^{-1} \circ \Phi \circ F$.
We realise that since $G$ is quadratic, there are two ways in which we could define our map $\widehat{\Phi}$, depending on which branch of the inverse we take. In the previous chapter, we constructed the map $\widehat{\Phi}$ in Proposition~\ref{newPhi}, starting off with the observation that $\phi$ was a conjugacy on the stars $X_F$ and $X_G$, and so we could set $\widehat{\Phi}$ to equal $\phi$ on $X_F$. It was then a fortunate consequence of the fixed cluster case that this restriction still allowed us to construct the homeomorphism $\widehat{\Phi}$ which satisfied all the required properties. We are not so fortunate in the case where there is more than one cluster. In this case, it may be that, \emph{a priori}, the map $\widehat{\Phi}$ we construct that satisfies $\widehat{\Phi} = G^{-1} \circ \Phi \circ F$ may not satisfy $\widehat{\Phi}|_{(X^1_F \cup X^2_F)} = \phi$. However, we notice that it is at least possible to set $\widehat{\Phi}|_{X_F^1} = \Phi|_{X_F^1}$, since 
\begin{align*}
	z \in X_F^1 \, \Longrightarrow 	& \,	F(z) \in X_F^2 \\
									\Longrightarrow	&	\, 	\Phi(F(z)) = \phi(F(z)) \in X_G^2 \\
									\Longrightarrow	& \, 	G^{-1} (\phi(F(z)) \in G^{-1} \left( X_G^2 \right).
\end{align*}

The set $G^{-1}(X_G^2)$ contains the star $X_G^1$ and the non-periodic pre-image stars of $X_G^2$. Of course, we are at liberty to pick the lift (the branch of $G^{-1}$) that gives us $\widehat{\Phi}(X_F^1) = X_G^1$. This choice will uniquely define our choice of $G^{-1}$ and so our map $\widehat{\Phi}$. However, this choice may not give us $\widehat{\Phi}(X_F^2) = X_G^2$, but instead will map it to the pre-image star (or pre-star, for short), and so our choice of pair $(\Phi, \widehat{\Phi})$ may not satisfy the requirements for the homeomorphisms in Thurston's theorem. However, we will show that we can carry out some suitable modifications to get two homeomorphisms which do satisfy the requirements. 

The reader may be suspicious about the above claim that the conditions $\widehat{\Phi} = G^{-1} \circ \Phi \circ F$ and $\widehat{\Phi}(X_F^1) = X_G^1$ are enough to uniquely define the homeomorphism. We briefly explain why this is the case below. Let the critical points of $F$ be $c_1 \in X_F^1$ and $c_2 \in X_F^2$. Let $\gamma$ be a path between $v_1 = F(c_1)$, the critical value in $X_F^2$ and $v_2$, the critical value in $X_F^1$ and with $\gamma \cap ( X_F^1 \cup X_F^2 )=\varnothing$. Then $F^{-1}(\gamma)$ is made up of $2$ curves from $c_1$ to $c_2$, and these split the sphere into $2$ regions, which we label anticlockwise around the critical point $c_1$ by $\A_1$ and $\A_2$, chosen so that $X_F^1 \subset \A_1$. The map $F_i \colon \A_i \to \CC \setminus \gamma$, the restriction of $F$ to $\A_i$, is a homeomorphism for each $i$. This can be extended continuously to the boundary. With a similar argument we see that $\Phi(\gamma) =: \gamma'$ is a path between the critical points $v_1'$ and $v_2'$. Hence $G^{-1}(\gamma')$ splits the sphere into $2$ regions, which we similarly label anticlockwise by $\A_1'$ and $\A_d'$, with $\A_1'$ being the region component containing $X^1_G$. Again the restriction $G_i \colon \A_i' \to \CC \setminus \gamma'$ is a homeomorphism and it can be extended continuously to the boundary. We now define $\widehat{\Phi}$ by mapping $\A_i$ onto $\A_i'$ so that $\widehat{\Phi} = G^{-1} \circ \Phi \circ F$.

\begin{lem}
	$\widehat{\Phi}$ is a homeomorphism. 
\end{lem}

\begin{proof}
	The details of this proof are similar to Proposition~\ref{newPhi}.
\end{proof}

Let $\tau \subset \CC$ be a simple closed curve which is disjoint from the stars $X_G^1$ and $X_G^2$ and which intersects $\gamma'$ in only one place and so that the winding number of $\tau$ about the first cluster point $p_1'$ is $1$. Denote by $D_{\tau}$ the Dehn twist about this curve $\tau$ in the anticlockwise direction. The plan is to modify $\Phi$ to a new function $\Phi_1 = D^{\circ j}_{\tau} \circ \Phi $ so that the pair $\Phi_1$ and $\widehat{\Phi}_1 = G^{-1} \circ \Phi_1 \circ F$ equal $\phi$ on $X_F^1 \cup X_F^2$.

By construction we map $X_F^1$ onto $X_G^1$ under $\widehat{\Phi}$. If $\widehat{\Phi}(X_F^2) = X_G^2$ then we are done, so suppose not. So we have   $X_F^2 \subset \A_k$ and $X_G^2 \subset \A_{\ell}'$, with $\{k,\ell \} = \{1,2\}$. If $k \neq \ell$ then replace $\Phi$ with $\Phi_1 =  D^{\circ (\ell-k)}_{\tau} \circ \Phi$. $\Phi_1$ is a homeomorphism since it is the composition of two homeomorphisms. We get a new path $\gamma_1 = \Phi_1(\gamma)$ ($\gamma$ defined as before) and hence we can define the regions $\A_1''$ and $\A_2''$ as the connected components of the complement of $G^{-1}(\gamma_1)$ in the sphere. As before we label anticlockwise round the first critical point and set $X_G^1 \subset \A_1''$. Now define $\widehat{\Phi}_1 = G^{-1} \circ \Phi_1 \circ F$, mapping $\A_i$ onto $\A_i''$ and forming the homeomorphism in the usual way. Note we have $\Phi_1|_{X_F^1 \cup X_F^2} = \phi$. 

\begin{lem}\label{l:starsequality}
	$\widehat{\Phi}_1 (X_F^1) = X_G^1$ and $\widehat{\Phi}_1 (X_F^2) = X_G^2$. Furthermore, $\widehat{\Phi}_1 |_{X_F^1 \cup X_F^2} = \phi$.
\end{lem}

\begin{proof}
  The first equality is clear since by construction we have $X_F^1 \subset \A_1$ and $X_G^1 \subset \A_1''$. Also by assumption we have $X_F^2 \subset \A_k$. So all we need to show is that $X_G^2 \subset \A_k''$. But this is precisely what is guaranteed by modifying $\Phi$ to $\Phi_1$ and thus $\widehat{\Phi}$ to $\widehat{\Phi}_1$.

  In passing, we note that the integer $k - \ell$ is equal to $1$ or $-1$ (or 0, if we are in the case where no Dehn twist is required). In either case, since $X_G^2 \subset \A_\ell'$, we must get, after applying the Dehn twist, that $X_G^2 \subset \A_{\ell-(\ell - k)}'' = \A_k''$. Hence $\widehat{\Phi}_1 (X_F^2) = X_G^2$.

\begin{figure}[ht]
\begin{center}

\input{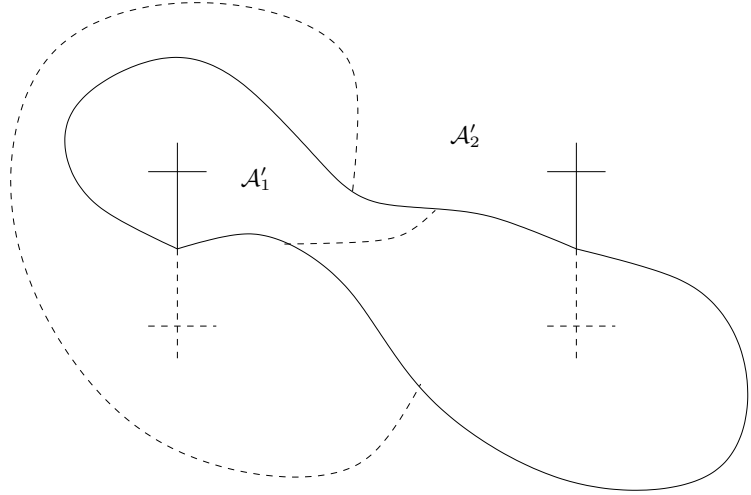}

\caption{Diagram for the proof of Lemma~\ref{l:starsequality}. $\A_i'$ is the image of $\A_i$ under $\widehat{\Phi}$. The dashed lines shows the effect of changing $\Phi$ to $D_{\tau} \circ \Phi$. Compare with Figure~\ref{f:twists3}.}
\label{f:twists}
\end{center}
\end{figure}	

\begin{figure}[ht]
\begin{center}

\input{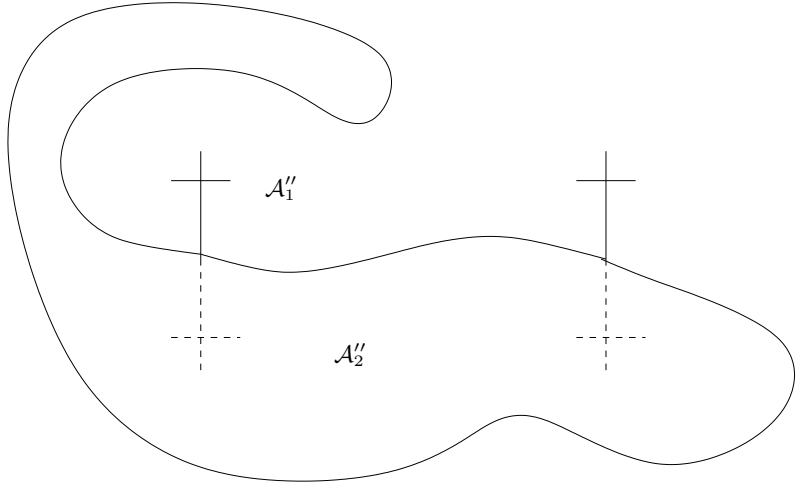}

\caption{The ``modified'' diagram from Figure~\ref{f:twists}, with the new regions $\A_i'' = \widehat{\Phi}_1(A_i)$ labelled.}
\label{f:twists3}
\end{center}
\end{figure}

It only remains to show $\widehat{\Phi}_1 |_{X_F^1 \cup X_F^2} = \phi$. But by construction we have
\begin{equation}\label{e:newphi}
 G \circ \widehat{\Phi}_1 = \Phi_1 \circ F.
\end{equation}
Define $\hat{\phi} = \widehat{\Phi}_1 |_{X_F^1 \cup X_F^2}$. Then we have
\begin{align*}
    G \circ \hat{\phi} 	=& \phi \circ F \quad \quad 	&\text{(by (\ref{e:newphi}))} \\
			=& G \circ \phi \quad \quad	&\text{(by (\ref{e:phi}))}
\end{align*}
on $X_F^1 \cup X_F^2$. Since $G$ is a homeomorphism on the stars, we get $\hat{\phi} = \phi$.
\end{proof}	


So we now have $\Phi_1$ and $\widehat{\Phi}_1$ agree on the stars. So we now compare the induced maps $\Psi_1$ and $\widehat{\Psi}_1$ from $A_F$ to $A_G$. Both these maps will equal $\psi$ on $\partial A_F$ and so $\widehat{\Psi}_1^{-1} \circ \Psi_1$ is a homeomorphism of $A_F$ which fixes the boundary pointwise. Since the mapping class group of the annulus is $\Z$, this homeomorphism is isotopic to $D^{\circ k}$ for some $k \in \mathbb{Z}$, where $D$ is the Dehn twist around the (anticlockwise) core curve of the annulus $A_F$. We can think of this core curve $C$ as being the pre-image under $\eta_F$ of some curve $\kappa$ separating the stars in the $F$-sphere. Note that the curve $\kappa' = F^{-1}(\kappa)$ maps onto $\kappa$ by a two to one covering. Also, the curve $C' = \eta_F^{-1}(\kappa')$ is homotopic to the curve $C$. We are now ready to prove Theorem~\ref{Per2case}.

\begin{proof}[Proof of Theorem~\ref{Per2case}]
	We begin by remarking that the homeomorphisms $\Phi_1$ and $\widehat{\Phi}_1$ satisfy $G \circ \widehat{\Phi}_1 = \Phi_1 \circ F$ and agree on $X_F^1 \cup X_F^2$ (and hence on $P_F$). So all that remains is to modify them so that these two conditions are preserved and that furthermore they are isotopic to one another. This is equivalent to making sure some suitable modification of $\Psi_1$ and $\widehat{\Psi}_2$ are isotopic to one another. It should be borne in mind that the definitions of $\Phi_2$ and $\widehat{\Phi}_2$ rely on each other, since we require $\Phi_i \circ F = G \circ \widehat{\Phi}_i$ for $i = 1,2$. Hence modifying one will force the modification of the other.

	Since $\kappa'$ maps to $\kappa$ in a two to one covering, a Dehn twist around $C'$ will correspond to the second power of a Dehn twist around $C$, again in light of the fact that $\Phi_i \circ F = G \circ \widehat{\Phi}_i$ and because $F$ and $G$ are degree 2. Working on the annulus $A_F$, we define $\widehat{\Psi}_2 = \widehat{\Psi}_1 \circ D^{\circ (-k)}_{C'}$ and $\Psi_2 = \Psi_1 \circ D^{\circ (-2k)}_C$. Since $C$ and $C'$ are homotopic, the Dehn twist around them has the same effect on the element of the mapping class group, hence we drop the subscript from now on. So we calculate
\begin{align*}
 	\widehat{\Psi}^{-1}_2 \circ \Psi_2 \, 	=& \, D^{\circ k} \circ \widehat{\Psi}_1^{-1} \circ \Psi_1 \circ D^{\circ (-2k)} \\
						\cong& \, D^{\circ k} \circ D^{\circ k} \circ D^{\circ (-2k)} \\
						=& \, \mathrm{Id}.
\end{align*}
Hence $\widehat{\Psi}_2$ and $\Psi_2$ are isotopic on the annulus and hence the maps $\widehat{\Phi}_2$ and $\Phi_2$ which are obtained by passing forward onto the Riemann sphere (using the maps $\tilde{\eta}_F$ and $\tilde{\eta}_G)$ satisfy the conditions for the homeomorphisms in Thurston's theorem. Hence $F$ and $G$ are Thurston equivalent. \end{proof}

\subsection{Differences in the higher degree case}

We will now briefly explain why the above method will not work when the degree is strictly greater than 2. A probable counterexample to the statment is given in the next section. In the proof above, we would be able to follow the method of this proof in the general degree $d$ case right until the proof of Lemma~\ref{l:starsequality}. In other words, the problem occurs when we want to make the two homeomorphisms isotopic rel $P_F$. Thinking of $\Phi$ as the homeomorphism on the range and $\widehat{\Phi}$ as the homeomorphism on the domain, we see that if we carry out one Dehn twist in the domain, then we will need to carry out the Dehn twist $d$ times in the range. We require the use of Dehn twists to ``undo'' the difference between the two homeomorphisms. The problem occurs if the difference is not a multiple of $d-1$ (which it always will be in the case $d=2$, as in the proof above).

To give an example, suppose we are in the degree $4$ case and that we found  $\widehat{\Psi}_1^{-1} \circ \Psi_1$ was isotopic to $D$. Then, if we carry out $k$ twists in the domain we carry out $4k$ twists in the range, and we see that we can solve for the number of twists required to make  $\widehat{\Psi}_1^{-1}$ and $\Psi_1$ isotopic by solving the equation $1 = 4k - k = 3k$. But this does not have an integer solution, and so we cannot carry out a (power of a) Dehn twist to correct the discrepancy. Hence the above technique would not supply a proof of equivalence in the higher degree case.

\section{Counterexamples in the general period two case}\label{s:higher}

In this section, we give an example which suggests that the generalisation of Theorem~\ref{Per2case} to the higher degree case is not possible, which would provide a stark interest with the fixed cluster case given in Theorem~\ref{Fixedcase}. We will construct two degree three rational maps using the mating operation, which have the same combinatorial data are not (according to calculations) equivalent in the sense of Thurston. We will not define the mating operation here, the interested reader is directed towards \cite{Milnor:mating,TanShish,Matingspaper} and \cite{TanLei:1990} or  for more information. Let $f_1(z) = z^3 + c_1$ be the map corresponding to the parameter $c_1 = 0.209925\ldots + 1.09351\ldots i$, and let $f_2(z) = z^3 + c_2$ be the map corresponding to the parameter $c_2 = -0.243965\ldots + 1.32241\ldots i$. Also, let $g_1(z) = z^3 + c'_1$ be the map corresponding to the parameter $c'_1 = -0.209925\ldots + 1.09351\ldots i$, and let $g_2(z) = z^3 + c'_2$ be the map corresponding to $c'_2 =  0.601679\ldots - 0.684721\ldots i$. Equivalently, if the reader is familiar with specifying polynomials by the parameter rays landing at the principal root points of their hyperbolic components, we have $f_1$ corresponds to the angles $(11/80,19/80)$, $f_2$ to $(22/80,24/80)$, $g_1$ to $(21/80,29/80)$ and $g_2$ to $(71/80,73/80)$. See Figure~\ref{f:parammaps} for the positions of the maps in parameter space, and Figure~\ref{f:juliasets} for the Julia sets of the maps, along with the external rays which land on the points which will become the cluster points in the mating. Note that for $f_2$, the external rays shown land at non-prinicpal root points of the critical orbit Fatou components.

\begin{figure}[ht]
  \centering
    \subfigure[The position of the maps $f_1$ and $f_2$.]{\includegraphics[width=0.44\textwidth]{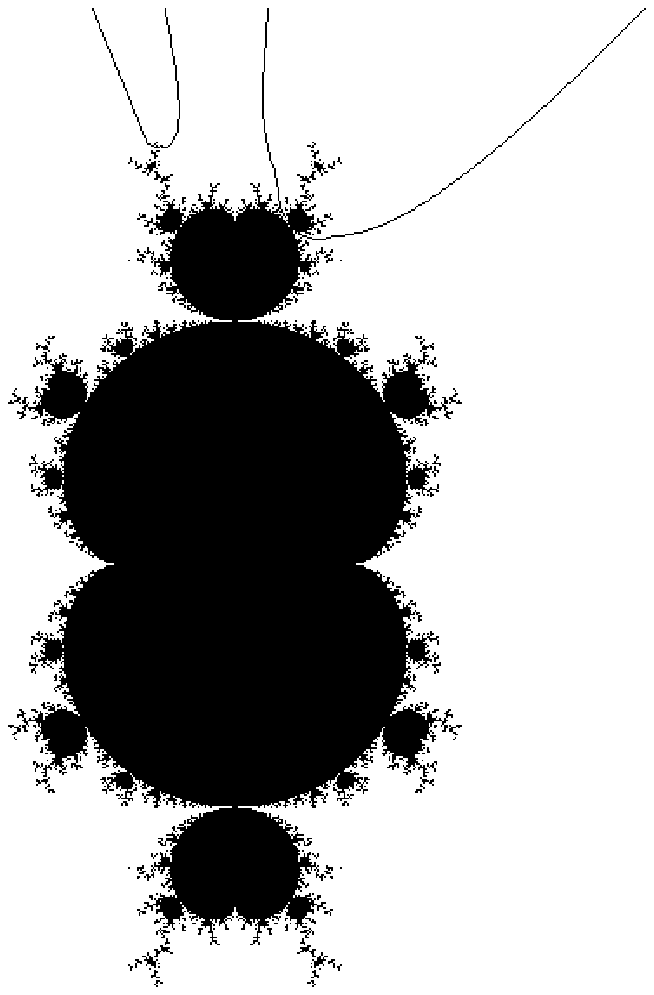}}
    \subfigure[The position of the maps $g_1$ and $g_2$.]{\includegraphics[width=0.44\textwidth]{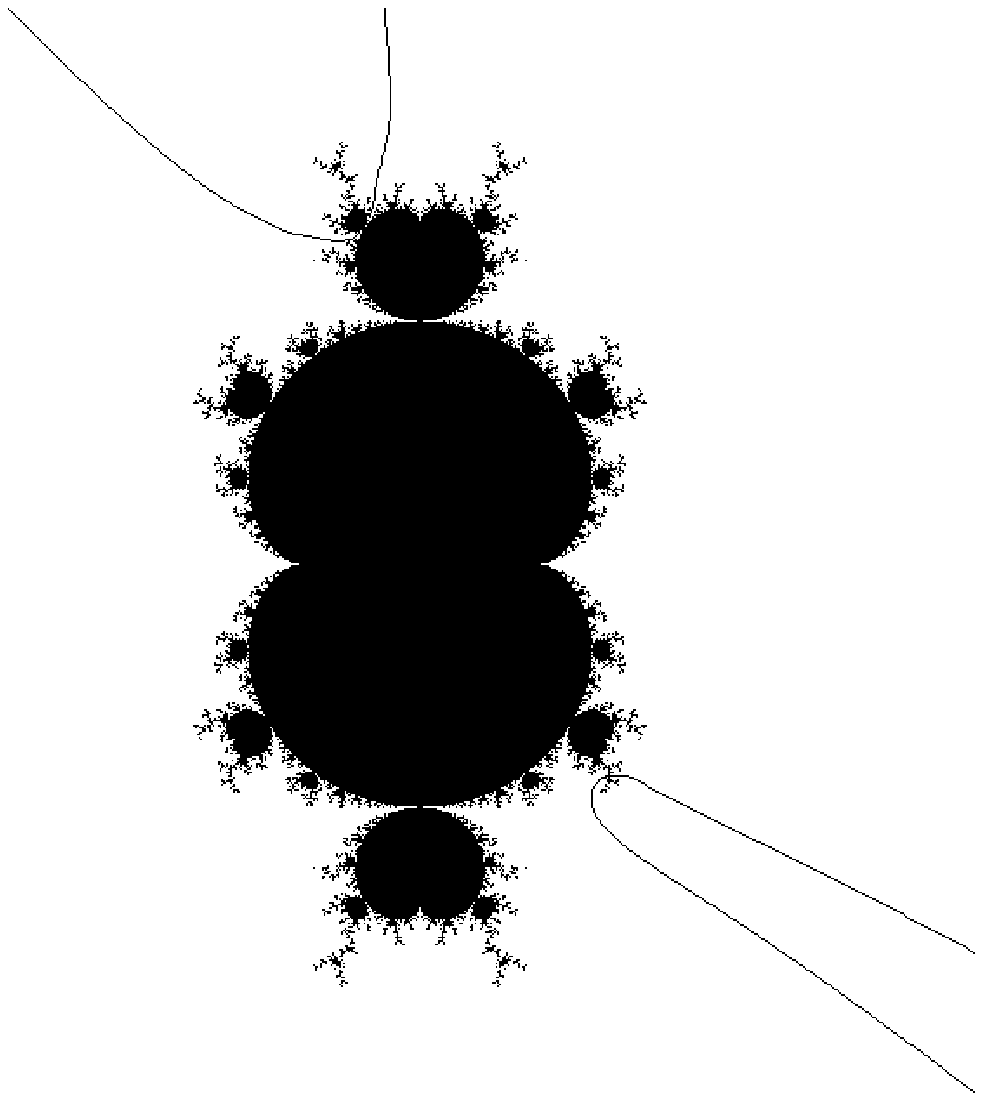}}
  \caption{Parameter space picture for the maps $f_1$, $f_2$, $g_1$ and $g_2$.}
  \label{f:parammaps}
\end{figure}

\begin{figure}[ht]
  \centering
    \subfigure[The Julia set of $f_1$.]{\includegraphics[width=0.43\textwidth]{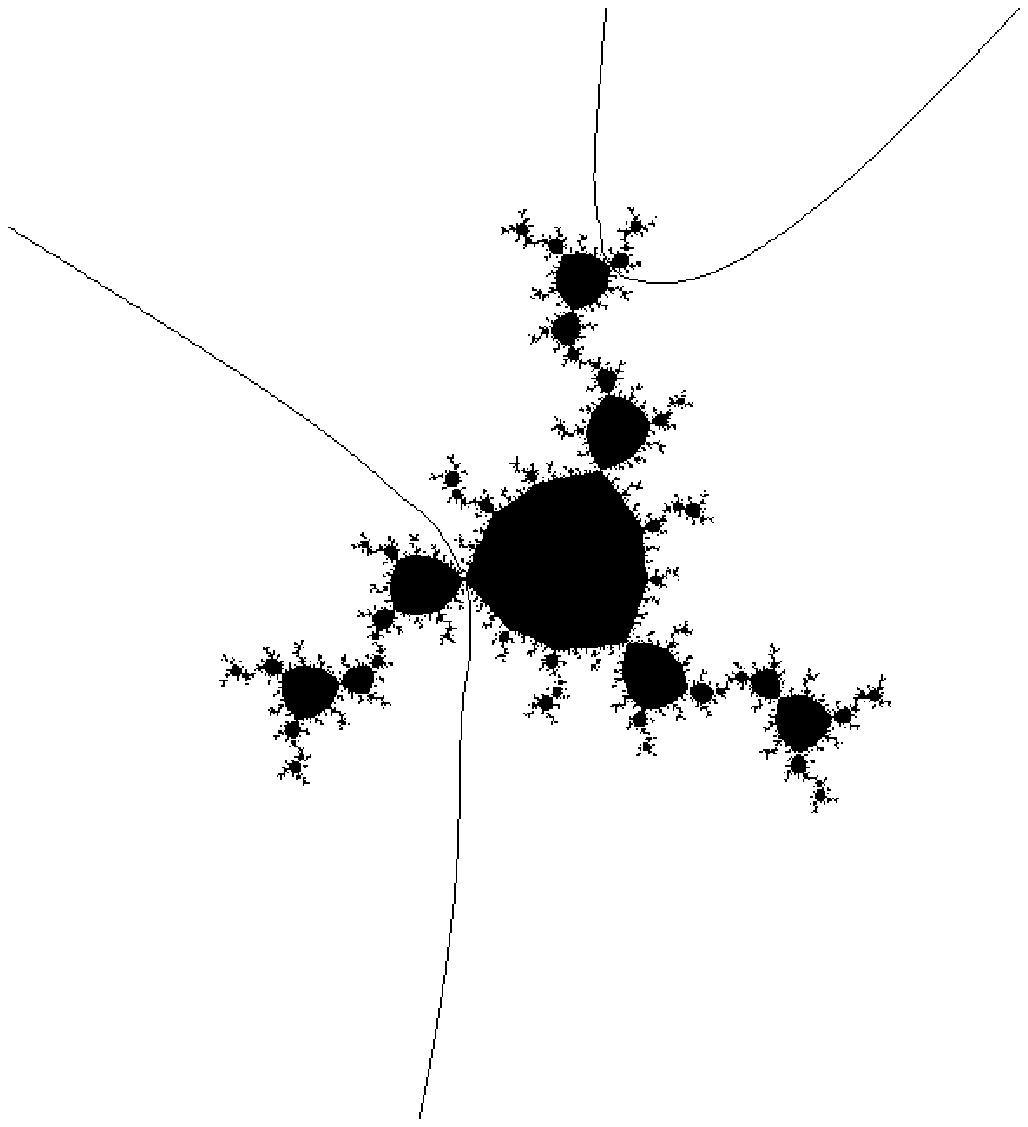}}
    \subfigure[The Julia set of $f_2$.]{\includegraphics[width=0.43\textwidth]{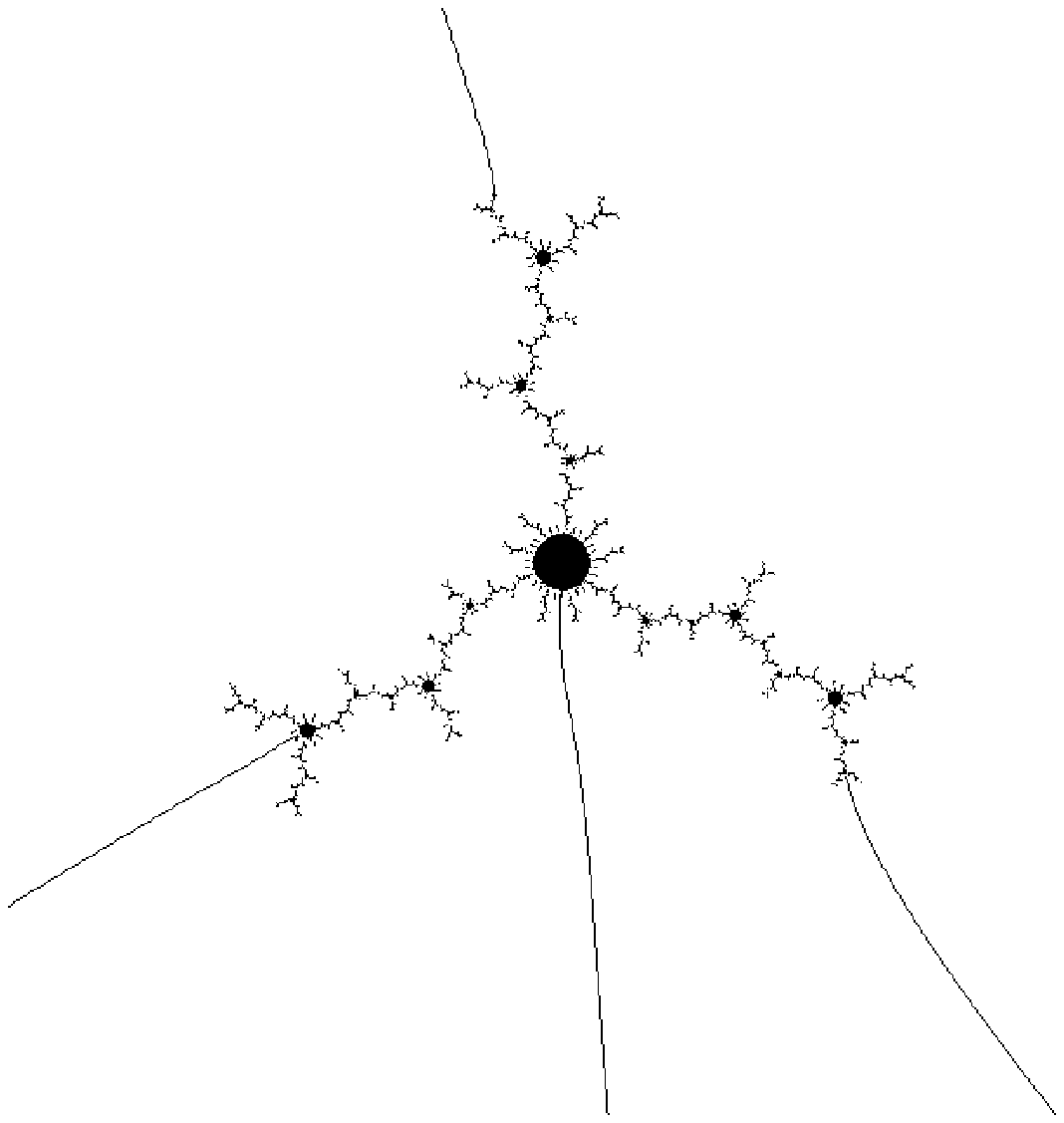}}
    \subfigure[The Julia set of $g_1$.]{\includegraphics[width=0.43\textwidth]{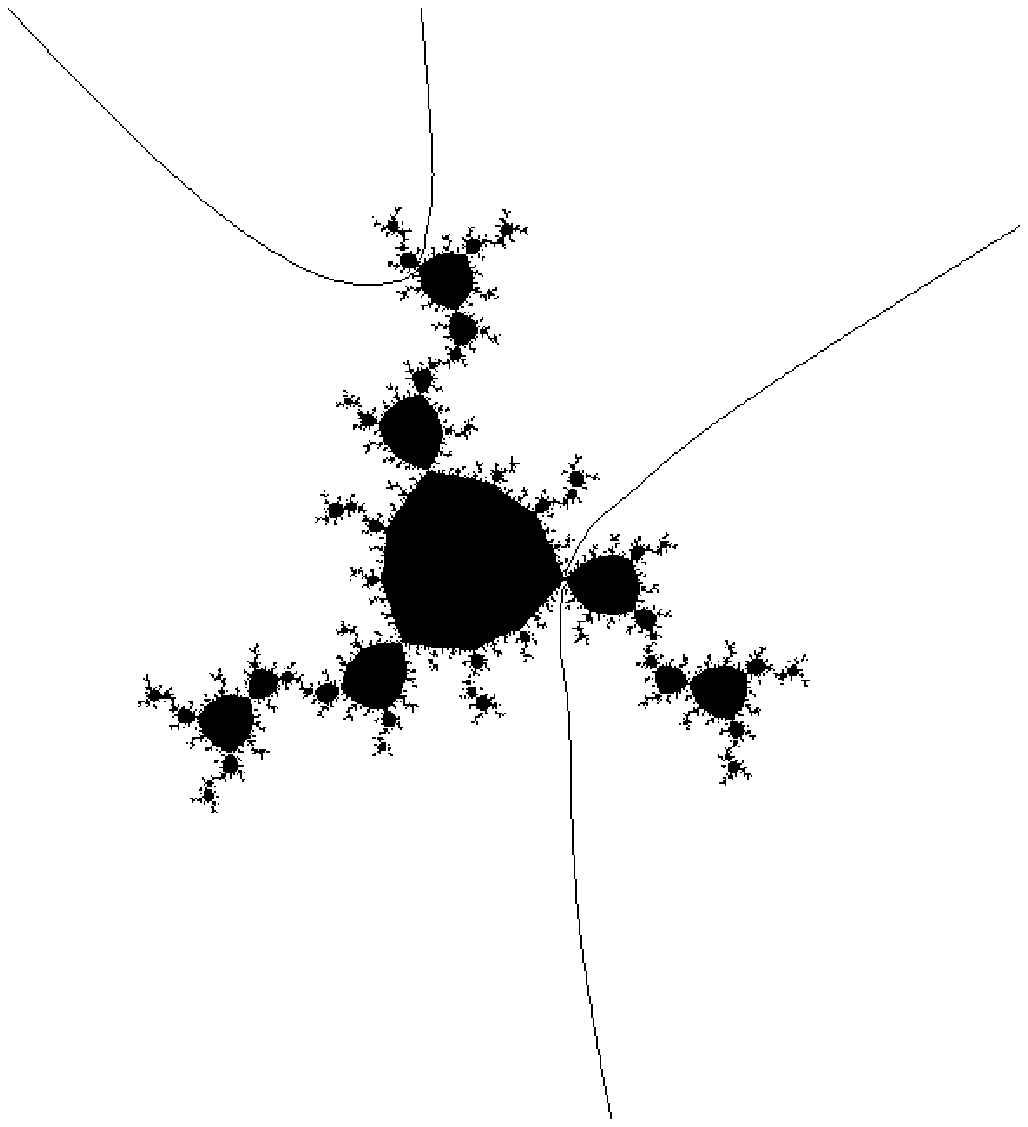}}
    \subfigure[The Julia set of $g_2$.]{\includegraphics[width=0.43\textwidth]{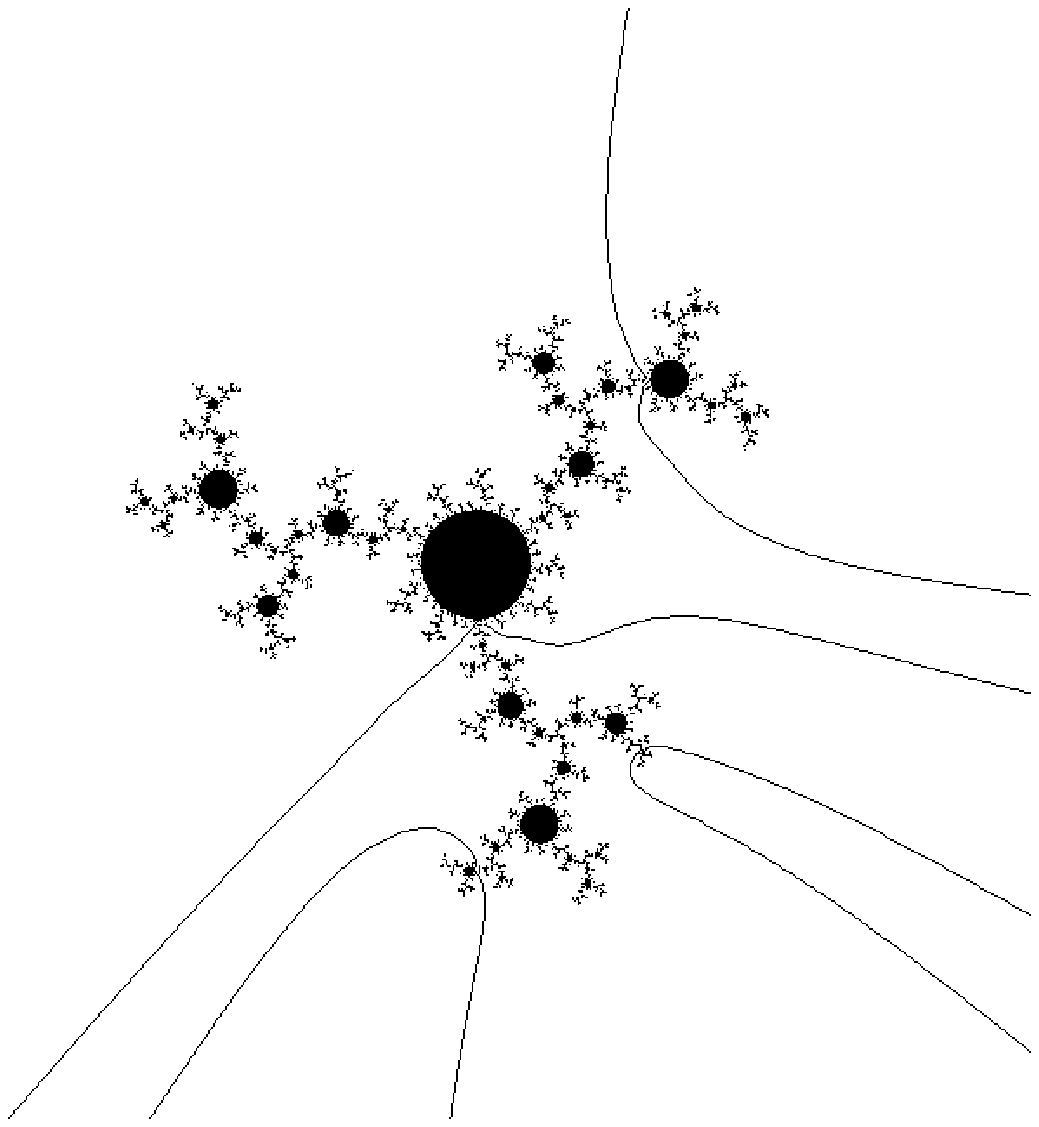}}
  \caption{The Julia sets of $f_1$, $f_2$, $g_1$ and $g_2$.}
  \label{f:juliasets}
\end{figure}

Now, construct the (topological) matings $F \cong f_1 \Perp f_2$ and $G \cong g_1 \Perp g_2$. Both the rational maps $F$ and $G$ have a period two cluster cycle which has rotation number $\rho = 1/2$ and critical displacement $\delta = 3$. However, these two maps are not Thurston equivalent. Calculations using the FR (written by Laurent Bartholdi) in GAP  suggest the two rational maps are given by
\[
  F(z) = \frac{(2.52260\ldots + 1.43040\ldots i) z^3 + 1}{(-4.31748\ldots-7.21673\ldots i)z^3 + 1}
\]
and
\[
  G(z) = \frac{(1.02505\ldots + 2.73636\ldots i) z^3 + 1}{(-6.43698\ldots+5.60985\ldots i)z^3 + 1}
\]
It is hoped that the confirmation that the maps $F$ and $G$ will be shown in a forthcoming paper. Furthermore, it is hoped we can find exactly what extra data is required (in terms of combinatorial data extrinsic to the cluster cycle) to get a complete classification of the general period two cluster case. 

\vspace{12pt}
\noindent\emph{Acknowledgements.} I would like to thank my PhD supervisor, Dr.~Adam Epstein, for all the help he has contributed towards my research and, in particular, this article. Furthermore, my thanks to Professor Mary Rees and Professor Anthony Manning for providing a number of useful suggestions in the preparation of this manuscript. This research was funded by a grant from EPSRC.

\bibliographystyle{amsalpha}
\bibliography{papers}

\end{document}